\documentclass[a4paper]{amsart}
\usepackage{amsmath,amsthm,amssymb}
\usepackage[dvipdfmx]{graphicx}

\newtheorem{thm}{Theorem}[section]

\newtheorem{cor}[thm]{Corollary}
\newtheorem{prop}[thm]{Proposition}

\theoremstyle{definition}

\newtheorem{defn}[thm]{Definition}

\theoremstyle{remark}

\newtheorem{rem}[thm]{Remark}        

\numberwithin{equation}{section}

\newcommand{\Pro}{\mathcal{P}}
\newcommand{\R}{\mathbb{R}}

\newcommand{\m}{\tilde{m}}

\newcommand{\Ob}{\mathrm{ObsDiam}}

\DeclareMathOperator{\Diam}{Diam}

\usepackage[abbrev,nobysame]{amsrefs}

\makeatletter
\def\@makefnmark{%
\leavevmode
\raise.9ex\hbox{\check@mathfonts
\fontsize\sf@size\z@\normalfont%
\@thefnmark}%
}
\makeatother

\title{Coarse Ricci curvature on the space of probability measures}
\author{Yu Kitabeppu}
\thanks{Partly supported by the Grant-in-Aid for JSPS Fellows, The Ministry of Education, Culture, Sports, Science and Technology, Japan}
\begin{document}
\maketitle
 \begin{abstract}
  In this paper we study the coarse Ricci curvature on the space of probability measures on 
  a metric space. 
  We consider the $p$-coarse Ricci curvature for $p\geq 1$, 
  which is a slight generalization of the coarse Ricci curvature 
  defined by Ollivier. 
  We get a natural random walk on the $L^p$-Wasserstein space 
  if the underlying space has a random walk. 
  The infimum of the $p$-coarse Ricci curvature on the $L^p$-Wasserstein space coincides with 
  that with respect to the original random walk. 
  Considering a random walk as a map, we investigate the relation between 
  Gromov-Hausdorff convergence and the $p$-coarse Ricci curvature.  
  We also study the concentration of measure phenomenon 
  related to the coarse Ricci curvature. 
   \end{abstract}
\section{Introduction}
  Ollivier defined a notion of the coarse Ricci curvature on 
  a metric space with a random walk in \cite{O}. 
  In this paper we study a notion of the $p$-coarse Ricci curvature 
  and investigate it on the space of probability measures. 
  We call a metric space $(X,d)$ a \emph{Polish metric space} 
  if it is a complete separable metric space.  
  Let $(X,d,\{m_x\}_{x\in X})$ be a Polish metric space with a \emph{random walk}, 
  where a random walk 
  is a family of Borel probability measures parametrized by $x\in X$.  
  Suppose that $m_x\in\Pro_p(X)$ for any $x\in X$, $p\geq 1$, where $\Pro_p(X)$ is 
  the \emph{$L^p$-Wasserstein space} (see Definition \ref{LpWp}). 
  \begin{defn}\label{pcRic}
 Let $(X,d,\{m_x\}_{x\in X})$ be a Polish metric space with a random walk and $1\leq p\leq \infty$. 
 We define the \emph{$p$-coarse Ricci curvature} along $xy$ by 
 \begin{align}
  \kappa_p(x,y):=1-\frac{W_p(m_x,m_y)}{d(x,y)}\label{cRicdef}
 \end{align} 
 for distinct points $x,y\in X$, where $W_p$ is the $L^p$-Wasserstein metric 
 (see Definition \ref{LpWp}). 
\end{defn}
  Considering the random walk as a map from $X$ to $\Pro_p(X)$, 
  we define a map 
  $\m : \Pro_p(X)\ni\mu\mapsto \m_{\mu}\in\Pro(\Pro_p(X))$ as 
  \begin{align}
   \int_{\Pro_p(X)}f(\sigma)\, \m_{\mu}(d\sigma):=\int_X f(m_x)\, \mu(dx)\label{ext}
  \end{align}
  for any $f\in\mathcal{C}_b(\Pro_p(X))$, where $\mathcal{C}_{b}(\Pro_p(X))$ is the set of all bounded 
  continuous functions on $\Pro_p(X)$. 
  We prove the following theorem. 
  \begin{thm}\label{lipext}
   Let $\left(X,d,\{m_x\}_{x\in X}\right)$ be a metric space with a random walk. 
   Let $\kappa_p^{m},\,1\leq p<\infty$, be the $p$-coarse Ricci curvature 
   with respect to a random walk $\{m_x\}_{x\in X}$. 
   Then we have 
   \begin{align}
    \inf_{x,y\in X}\kappa_p^{m}(x,y)=\inf_{\mu,\nu\in\Pro_p(X)}\kappa_p^{\tilde{m}}(\mu,\nu).
    \label{lipexteq}
   \end{align} 
  \end{thm}  
  By Theorem \ref{lipext}, we have a contrasting 
  difference between the curvature-dimension condition (\cite{LV,St1,St2}) and 
  the $p$-coarse Ricci curvature. 
  In fact, Chodosh \cite{C} proved that the $L^2$-Wasserstein space over 
  a unit interval with an entropic measure 
  \emph{never} satisfy 
  the curvature-dimension condition $CD(K,\infty)$ for any $K\in\R$. 
  
  We observe what happens on the $L^p$-Wasserstein space provided that  
  a contraction semigroup converges to an invariant distribution. 
  We show that the convergence of the contraction semigroup to 
  a unique invariant distribution leads to 
  the convergence of the flow to the Dirac measure of the invariant distribution 
  and 
  the convergent rate coincides (see Remark \ref{convrate}). 
  
  It is important to investigate the relation between a version of a notion of Ricci curvature 
  and the Gromov-Hausdorff convergence. 
  We prove the stability of 
  the $p$-coarse Ricci curvature with respect to 
  the Gromov-Hausdorff convergence in the following sense. 
  \begin{thm}\label{GHtop}
   Let $\{(X_n,\ast_n,d_n,\{m^n_x\}_{x\in X_n})\}_{n\in \mathbb{N}}$ 
   be a sequence of locally compact, geodesic, Polish pointed metric spaces with random walks 
   such that $(X_n,\ast_n,d_n)\rightarrow (X,\ast,d)$ in the Gromov-Hausdorff sense. 
   Suppose that the following three conditions $(1)$, $(2)$ and $(3)$ are satisfied: 
   \begin{enumerate}
   \item There exists a constant $\kappa_0$ such that 
   $\inf_n\inf_{x,y\in X_n}\kappa^{m^n}_p(x,y)\geq\kappa_0$ holds; 
   \item For any sequence $\{x_n\in X_n\}_{n\in\mathbb{N}}$ with 
   $x_n\rightarrow x\in X$ (see Definition \ref{ptGH} below), and for any positive number 
   $\epsilon >0$, there exist compact sets $K^n_{\epsilon}\subset X_n$ with $x_n\in K^n_{\epsilon}$ 
   such that $m^n_{x_n}(X_n\setminus K^n_{\epsilon})\leq \epsilon$ for any $n\in\mathbb{N}$ and 
   $\sup_{n} Diam(K^n_{\epsilon})<\infty$ ;
   \item For any sequence $\{x_n\in X_n\}_{n\in\mathbb{N}}$ with 
   $x_n\rightarrow x\in X$, the uniform boundedness condition for $m^n_{x_n}$, 
   \begin{align}
    \sup_{n\in\mathbb{N}}\int_{X_n}d(\ast_n,y)^p\,m^n_{x_n}(dy)<\infty,\notag
   \end{align}
   is satisfied.
   \end{enumerate} 
   Then, there exist a subsequence $\{(X_{n_k},d_{n_k}, \{m^{n_k}_x\}_{x\in X_{n_k}})\}_k$ and 
   a random walk $\{m_x\}_{x\in X}$ on $X$ such that 
   $\{m^{n_k}_x\}_{x\in X_{n_k}}$ converges to $\{m_x\}_{x\in X}$ as a map 
   (see Definition \ref{convasmap}). 
   In particular, the $p$-coarse Ricci curvature with respect to $\{m_x\}_{x\in X}$ 
   satisfies 
   \begin{align}
    \inf_{x,y\in X}\kappa^m_p(x,y)\geq \kappa_0. \notag
   \end{align} 
   \end{thm} 
   \begin{rem}
    The condition $(2)$ in Theorem \ref{GHtop} is necessary. 
    For example, let $(X_n,\ast_n)=(\R,0)$. 
    We take the random walks $\{m^n_x\}_{x\in \R},$ $n=1,2,\ldots$, on $\R$ 
    of which each measure $m^n_x$ 
    is the uniform measure on the closed interval $[x+n,x+n+1]$. 
    Although the $p$-coarse Ricci curvature is always $0$ for $(\R,\{m^n_x\}_{x\in\R})$, 
    the sequence of random walks $\{m^n_x\}_{x\in\R},\;n=1,2\ldots$, does not converge.   
   \end{rem}
  Ollivier introduced a notion of the Gromov-Hausdorff convergence with random walks and 
  proved the stability of a lower bound of the coarse Ricci curvature with respect to that convergence 
   (see \cite{O}*{Definition 55 and Proposition 56}). 
  There he assumed the existence of the limit random walk on the 
  limit space. Our result claims the existence of the random walk and gives another proof 
  of the existence of a lower bound of the coarse Ricci curvature on the limit space. 
  Ollivier's result and Theorem \ref{GHtop} are equivalent to each other 
  if all of metric spaces in the sequence and the limit space are compact. 
  See Proposition \ref{OGH} for more precise information. 
  
  We also show a relation between 
  the concentration of measure phenomenon and a lower bound of a coarse Ricci curvature. 
  The concentration of measure phenomenon is deeply related to a lower bound of 
  the Ricci curvature bound in a Riemannian manifold \cite{FS,GM}. 
  \begin{thm}\label{Levy}
   Let $\{(X_n,d_n,\nu_n)\}_n$ be a L{\'e}vy family (see Definition \ref{levydef}). 
   Suppose that there exists random walks 
   $m_n : X_n\rightarrow\Pro_1(X_n)$ with a uniform lower bound of the 1-coarse Ricci curvature. 
   Then $\{(\Pro_1(X_n),W_1,(\tilde{m}_n)_{\nu_n})\}_n$ is also a L{\'e}vy family. 
  \end{thm}
\section{Preliminaries}
Let $(X,d)$ be a Polish metric space and $\mathcal{B}(X)$ the set of all Borel sets in $X$.   
We denote by $\Pro(X)$ the set of all Borel probability measures on $X$ equipped 
with the weak topology. The following well-known proposition is a characterization of the 
relative compactness of the family of probability measures. 
\begin{prop}[\cite{B}*{Theorem 5.1}]\label{proh}
 A family of probability measures $\{\mu_n\}_{n\in\mathbb{N}}$ is relatively compact in $\Pro(X)$ 
 if and only if 
 $\{\mu_n\}_{n\in\mathbb{N}}$ is tight, that is, for given $\epsilon>0$ there exists a compact subset 
 $K_{\epsilon}\subset X$ such that $\mu_n(X\setminus K_{\epsilon})\leq\epsilon$ for any 
 $n\in\mathbb{N}$. 
\end{prop}
We use the following proposition later. 
\begin{prop}[\cite{B}*{Theorem 2.1}]
 Let $\mu_n, n=1,2,3\ldots$ and $\mu$ be Borel probability measures on $X$ 
 such that $\mu_n$ converges to $\mu$ 
 weakly. Then for any closed subset $C\subset X$ and any open subset $G\subset X$, we have
 \begin{align}
  \limsup_{n\rightarrow 0}\mu_n(C)\leq \mu(C),\label{usc}\\
  \liminf_{n\rightarrow 0}\mu_n(G)\geq \mu(G).\label{lsc}
 \end{align} 
\end{prop}
For any $\mu,\nu\in\Pro(X)$, we call a measure $\pi\in\Pro(X\times X)$ a \emph{coupling} between 
$\mu$ and $\nu$ if 
\begin{align}
 \pi(A\times X)=\mu(A),\quad\pi(X\times A)=\nu(A)\quad \text {for any}\; A\in\mathcal{B}(X).
\end{align}
We denote by $\Pi(\mu,\nu)$ the set of all couplings between $\mu$ and $\nu$. 
We define a metric on $\Pro(X)$ 
that induces a topology stronger than the weak topology. 
\begin{defn}\label{LpWp}
 For $\mu,\nu\in\Pro(X)$, $1\leq p\leq \infty$, 
 the \emph{$L^p$-Wasserstein distance} $W_p(\mu,\nu)$ between $\mu$ and 
 $\nu$ is defined by 
 \begin{align}
  W_p(\mu,\nu)
  :=\inf\left\{\|d\|_{L^p(\pi)} ; \pi\in\Pi(\mu,\nu)\right\}.\label{Wp}
 \end{align}
 $W_p$ is finite on 
 \begin{align}
  \Pro_p(X):=\left\{\mu\in\Pro(X) ; \|d(o,\cdot)\|_{L^p(\mu)}<\infty\;\text{for some}\; o\in X\right\}.\notag
 \end{align}
\end{defn}
It is a known fact that the metric space $(\Pro_p(X),W_p)$ is a Polish metric space. 
It is compact if $X$ is compact. 
We characterize the convergence in $\Pro_p(X)$. 
\begin{thm}[\cite{V2}*{Definition 6.8, Theorem 6.9}]\label{equiv}
 Let $\mu_n,\,n=1,2,\ldots$, and $\mu$ be Borel probability measures on $X$. 
 The following conditions $(1)$-$(5)$ are all equivalent to each other: 
 \begin{enumerate}
  \item $W_p(\mu_n,\mu)\rightarrow 0$ as $n\rightarrow \infty$.
  \item $\mu_n\rightarrow \mu$ weakly as $n\rightarrow \infty$ and for some $x_0\in X$
  \begin{align}
  \limsup_{n\rightarrow\infty}\int d(x_0, x)^p\,\mu_n(dx)\leq \int d(x_0,x)^p\,\mu(dx). \label{Wpequiv}
  \end{align}
  \item $\mu_n\rightarrow \mu$ weakly as $n\rightarrow \infty$ and for some $x_0\in X$
  \begin{align}
  \int d(x_0,x)^p\,\mu_n(dx)\rightarrow \int d(x_0,x)^p\,\mu(dx). \notag
  \end{align}
  \item $\mu_n\rightarrow \mu$ weakly as $n\rightarrow \infty$ and for some $x_0\in X$
  \begin{align}
  \lim_{R\rightarrow\infty}\limsup_{n\rightarrow\infty}\int_{d(x_o,x)\geq R}
  d(x_0,x)^p\,\mu_n(dx)=0.\notag 
  \end{align}
  \item For all continuous functions $\phi$ with $|\phi(x)|\leq C(1+d(x_0,x)^p)$, $C\in\R$, one has \\
  \begin{align}
  \int \phi(x)\,\mu_n(dx)\rightarrow\int\phi(x)\,\mu(dx). \notag  
  \end{align}
 \end{enumerate}
\end{thm}
The following theorem is well known. 
\begin{thm}[Kantorovich duality, \cite{V}*{Theorem 1.3}]
 Let $\mu, \nu\in\Pro(X)$ and $\Phi_{d^p}(\mu,\nu)$ be the set of all pairs of measurable functions 
 $(\phi, \psi)\in L^1(\mu)\times L^1(\nu)$ satisfying 
 \begin{align}
  \phi(x)+\psi(y)\leq d^p(x,y)\notag
 \end{align} 
 for $\mu$-a.e. $x\in X$ and $\nu$-a.e. $y\in X$. 
 Then we have 
 \begin{align}
  W_p(\mu,\nu)^p
  =\sup_{(\phi,\psi)\in\Phi_{d^p}(\mu,\nu)}\left\{\int_X\phi\, d\mu+\int_X\psi\, d\nu \right\}.\label{Kd}
 \end{align}
\end{thm}
\begin{rem}
 If $p=1$, we may replace $(\phi,\psi)\in\Phi_{d}$ by $(\phi,-\phi)$, where 
 $\phi$ is a 1-Lipschitz function \cite{V}. 
\end{rem}
Assume that there exists a lower bound $\kappa_0\in\mathbb{R}$ of the $p$-coarse Ricci curvature 
for a metric space with a random walk $(X,d,\{m_x\}_{x\in X})$. 
Then by (\ref{cRicdef}), we get $W_p(m_x,m_y)\leq (1-\kappa_0)d(x,y)$ 
(also see Proposition \ref{convo}). 
This means that the map $m : X\rightarrow\Pro_p(X),\; x\mapsto m_x$,  
is a $(1-\kappa_0)$-Lipschitz map. 
This point of view is very important in this paper. 

Set $\mu\ast m(dx):=\int_Xm_x(dy)\mu(dx)$ for $\mu\in\Pro(X)$. 
We call a measure $\nu\in\Pro(X)$ an \emph{invariant measure} if $\nu=\nu\ast m$. 
An invariant measure $\nu$ is \emph{reversible} if $m_x(dy)\nu(dx)=m_y(dx)\nu(dy)$.  
The following properties are useful. 
\begin{prop}[\cite{O}*{Proposition 20}]\label{convo}
 Let $(X,d,\{m_x\}_{x\in X})$ be a metric space with a random walk, $\kappa_0$ a real number 
 and $1\leq p<\infty$. 
 Then, $\inf_{x,y}\kappa_p(x,y)\geq\kappa_0$ if and only if 
 \begin{align}
  W_p(\mu\ast m,\nu\ast m)\leq (1-\kappa_0)W_p(\mu,\nu)\notag
 \end{align} 
 holds for any $\mu,\nu\in\Pro_p(X)$.  
\end{prop}
\begin{cor}[\cite{O}*{Corollary 21}]\label{Ocor}
 Let $(X,d,\{m_x\}_{x\in X})$ be a metric space with a random walk and $1\leq p<\infty$. 
 Assume that the 
 $p$-coarse Ricci curvature satisfies $\kappa_p(x,y)\geq \kappa_0>0$ for any $x,y\in X$ and for 
 a constant $\kappa_0$. 
 Then there exists a unique invariant measure $\nu\in\Pro_p(X)$.
\end{cor}
\section{Extension of Lipschitz maps}\label{seclip}
\begin{proof}[Proof of Theorem \ref{lipext}]
 We assume that both $\inf_{x,y}\kappa_p^m(x,y)$ and 
 $\inf_{\mu,\nu}\kappa_p^{\tilde{m}}(\mu,\nu)$ are finite value. 
 Let $\kappa_0\in\mathbb{R}$ such that $\inf_{x,y}\kappa_p^m(x,y)=\kappa_0$. 
 Set $C:=1-\kappa_0$. We know that the $p$-coarse Ricci curvature is bounded below by $\kappa_0$ 
 if and only if the map $m : X\rightarrow \Pro_p(X)$ is $C$-Lipschitz.  
 We prove that the map $\m$ is a $C$-Lipschitz map from $\Pro_p(X)$ to $\Pro_p(\Pro_p(X))$ 
 to show $\inf_{\mu,\nu}\kappa^{\tilde{m}}(\mu,\nu)\geq \inf_{x,y}\kappa^m(x,y)$. 
 Let $\mu, \nu\in\Pro_p(X)$ and $(\phi,\psi)\in\Phi_{W_p^p}(\m_{\mu},\m_{\nu})$. 
 Since the map $m$ is a $C$-Lipschitz function, the image of $m$, $m(X)$, 
 is a universal measurable set on $\Pro(\Pro_p(X))$ (\cite{Bo}*{Theorem 7.4.1}). 
 Then we obtain $\m_{\mu}(m(X))=\int_X\mu(dx)=1$ and $\m_{\nu}(m(X))=1$. 
 From the above claim, the inequality  
 \begin{align}
  \phi(m_x)+\psi(m_y)\leq W_p^p(m_x, m_y)\leq C^pd^p(x,y)\label{mble}
 \end{align}
 holds for $\mu$-almost every $x\in X$ and $\nu$-almost every $y\in X$. 
 By (\ref{mble}), we have $(\phi\circ m/C^p,\psi\circ m/C^p)\in\Phi_{d^p}(\mu,\nu)$. 
 Then by using (\ref{Kd}) we have 
 \begin{align}
  &\int_{\Pro_p(X)}\phi(\sigma)\,\m_{\mu}(d\sigma)
  +\int_{\Pro_p(X)}\psi(\sigma)\,\m_{\nu}(d\sigma)
  =\int_X\phi(m_x)\,\mu(dx)+\int_X\psi(m_x)\,\nu(dx)\notag\\
  &=C^p\left\{\int_X\frac{\phi(m_x)}{C^p}\,\mu(dx)+\int_X\frac{\psi(m_x)}{C^p}\,\nu(dx)\right\}\notag\\
  &\leq C^pW_p^p(\mu,\nu)<\infty.\notag
 \end{align}
 Taking the supremum over all $(\phi,\psi)\in\Phi_{W_p^p}(\m_{\mu},\m_{\nu})$, 
 we get $W_p(\m_{\mu},\m_{\nu})\leq CW_p(\mu,\nu)$, which implies 
 $\tilde{m}_{\mu}\in\Pro_p(\Pro_p(X))$ for any $\mu\in\Pro_p(X)$ and 
 $\inf\kappa^{\tilde{m}}(\mu,\nu)\geq \inf\kappa^m(x,y)$. 
 
 It is easy to prove the converse implication. 
 Indeed, We assume $\inf_{\mu,\nu}\kappa_p^{\tilde{m}}(\mu,\nu)=\kappa_0$. 
 Since $\tilde{m}_{\delta_x}=m_x$ for any $x\in X$, we have 
 \begin{align}
  W_p(m_x,m_y)=W_p(\tilde{m}_{\delta_x},\tilde{m}_{\delta_y})
  \leq (1-\kappa_0)W_p(\delta_x,\delta_y)=(1-\kappa_0)d(x,y).\notag
 \end{align}
 When either $\inf_{x,y}\kappa_p^m(x,y)$ or $\inf_{\mu,\nu}\kappa_p^{\tilde{m}}(\mu,\nu)$ is not finite, 
 the above arguments lead infiniteness of the other. In this sense, we get (\ref{lipexteq}). 
\end{proof}
\section{Convergence to invariant measure}\label{ex}
We first show a relation between the invariant distribution of 
$\{\tilde{m}_{\mu}\}_{\mu}$ on $\Pro_p(X)$ and of $\{m_x\}_{x\in X}$ on $X$. 
\begin{prop}\label{inv}
 Let $(X,d,\{m_x\}_{x\in X})$ be a metric space with a random walk and $\nu$ an invariant distribution 
 for $\{m_x\}_{x\in X}$. Then $\tilde{m}_{\nu}$ is an invariant distribution 
 for $\{\tilde{m}_{\mu}\}_{\mu\in\Pro_p(X)}$. 
 Moreover, if $\nu$ is reversible, then $\m_{\nu}$ is also reversible. 
\end{prop}
\begin{proof}
 For any $f\in\mathcal{C}_b(\Pro_p(X))$ we have 
 \begin{align}
  &\int_{\Pro_p(X)}\int_{\Pro_p(X)}f(\sigma)\,\tilde{m}_{\mu}(d\sigma)\tilde{m}_{\nu}(d\mu)\notag\\
  &=\int_X\int_{\Pro_p(X)}f(\sigma)\,\tilde{m}_{m_x}(d\sigma)\nu(dx)\notag\\
  &=\int_X\int_Xf(m_y)\,m_x(dy)\nu(dx)\notag\\
  &=\int_Xf(m_x)\,\nu(dx)\notag\\
  &=\int_{\Pro_p(X)}f(\sigma)\,\tilde{m}_{\nu}(d\sigma).\notag
 \end{align}
 This shows that $\tilde{m}_{\nu}$ is an invariant measure for $\{\tilde{m}_{\mu}\}_{\mu}$.
 
 Supposing $\nu$ is reversible, we prove the reversibility of $\tilde{m}_{\nu}$. 
 We denote $\tilde{m}_{\nu}$ by $\tilde{\nu}$ for simplicity.  
 We have
 \begin{align}
  \int_{\Pro_p(X)}\int_{\Pro_p(X)}f(\sigma)g(\tau)\, \tilde{m}_{\sigma}(d\tau)\tilde{\nu}(d\sigma)
  &=\int_X\int_{\Pro_p(X)}f(m_x)g(\tau)\, \tilde{m}_{m_x}(d\tau)\nu(dx)\notag\\
  &=\int_X\int_Xf(m_x)g(m_y)\, m_x(dy)\nu(dx)\notag\\
  &=\int_X\int_Xf(m_x)g(m_y)\, m_y(dx)\nu(dy)\notag\\
  &=\int_X\int_{\Pro_p(X)}f(\sigma)g(m_y)\, \tilde{m}_{m_y}(d\sigma)\nu(dy)\notag\\
  &=\int_{\Pro_p(X)}\int_{\Pro_p(X)}f(\sigma)g(\tau)\, \tilde{m}_{\tau}(d\sigma)\tilde{\nu}(d\tau)\notag
 \end{align}
 for any $f,g \in\mathcal{C}_b(\Pro_p(X))$. This completes the proof.
\end{proof}
Fix $1\leq p<\infty$. Let $(X,d)$ be a compact metric space 
and $\{m_x^t\}_{x\in X,\;t>0}\subset\Pro_p(X)$ a familly of random walks. 
Assume that $\{m_x^t\}_{x\in X,\;t>0}$ is a \emph{contraction semigroup}, i.e., 
\begin{align}
m^{s+t}_x&=m^t_x\ast m_x^s,\notag\\
W_p(\sigma\ast m^t,\tau\ast m^t)&\leq f(t)W_p(\sigma,\tau)\label{contract}
\end{align}
holds for any $x,y\in X,\;t,s>0$, where $f : (0,\infty)\rightarrow (0,1)$ 
is a non-increasing function such that $f(t)\rightarrow 0$ as $t\rightarrow \infty$. 
By Proposition \ref{convo}, 
we have $W_p(\sigma\ast m^{t+s},\tau\ast m^{t+s})\leq f(t)W_p(\sigma\ast m^s,\tau\ast m^s)$.  
The contraction property (\ref{contract}) yields the existence of a unique invariant distribution $\nu$. 
We have a random walk $\{\m^t_{\mu}\}_{\mu\in\Pro_p(X)}$ as in (\ref{ext}). 
By Proposition \ref{inv}, we have a unique invariant distribution $\tilde{\nu}^t$ 
for $\{\m^t_{\mu}\}_{\mu\in\Pro_p(X)}$. 
\begin{prop}\label{ratelem}
 The measures $\tilde{\nu}^t$ converges to $\delta_{\nu}$ on $\Pro_p(X)$ 
 as $t\rightarrow \infty$. 
\end{prop}
\begin{proof}
 Take $s>t>0$. Let $(\phi,\psi)\in\Phi_{W_p^p}(\tilde{\nu}^t,\tilde{\nu}^s)$. 
 By Proposition \ref{inv}, $\tilde{\nu}^t=\m^t_{\nu}$ for any $t>0$. 
 Since supp $\tilde{\nu}^t=m^t(\text{supp}\, \nu)$ and 
 supp $\tilde{\nu}^s=m^s(\text{supp}\, \nu)$, 
 it is clear that $\phi(m^t_x)+\psi(m^s_y)\leq W^p_p(m^t_x,m^s_y)$ holds 
 for $\nu$-a.e.~$x\in X$ and $\nu$-a.e.~$y\in X$. 
 Then we have 
 \begin{align}
  &\int_{\Pro_p(X)}\phi(\sigma)\, \tilde{\nu}^t(d\sigma)
  +\int_{\Pro_p(X)}\psi(\sigma)\, \tilde{\nu}^s(d\sigma)\notag\\
  &=\int_X\phi(m^t_x)\,\nu(dx)+\int_X\psi(m^s_x)\,\nu(dx)\notag\\
  &=\int_X\phi(m_x^t)+\psi(m_x^s)\, \nu(dx)\notag\\
  &\leq \int_X W^p_p(m_x^t, m_x^s)\, \nu(dx)\notag\\
  &\leq f(t)^p\int_X W^p_p(\delta_x, m_x^{s-t})\, \nu(dx)\notag\\
  &\leq \Diam (X)^pf(t)^p \rightarrow 0\quad \text{as} \quad t\rightarrow \infty.\notag
 \end{align}
 By the completeness of $\Pro_p(X)$ and the Kantorovich duality (\ref{Kd}), 
 we get $W_p(\tilde{\nu}^t,\tilde{\nu}^s)\rightarrow 0$ as $t\rightarrow \infty$. 
 Then $\{\tilde{\nu}^t\}_{t>0}$ is a convergent family. 
The limit measure of $\{\tilde{\nu}^t\}_{t>0}$ denotes by $\tilde{\nu}$. 
 Let $B$ be an arbitrary closed set in $\Pro_p(X)$ with $\nu\in B$. 
 By using Fatou's Lemma and (\ref{usc}), we have
 \begin{align}
  \tilde{\nu}(B)
  &\geq \limsup_{t\rightarrow \infty}\tilde{\nu}^t(B)\notag\\
  &\geq \liminf_{t\rightarrow \infty}\int_{\Pro_p(X)}\chi_B(\sigma)\, \tilde{\nu}^t(d\sigma)\notag\\
  &=\liminf_{t\rightarrow \infty}\int_X\chi_B(m^t_x)\, \nu(dx)\notag\\
  &\geq \int_X\liminf_{t\rightarrow \infty}\chi_B(m^t_x)\, \nu(dx)\notag\\
  &=1,\notag
 \end{align}
 which implies $\tilde{\nu}=\delta_{\nu}$. 
\end{proof}
\begin{rem}\label{convrate}
 By the proof of Proposition \ref{ratelem}, we see that $\tilde{\nu}^t$ 
 converges to $\delta_{\nu}$ with the same rate for $m^t_x$ converging to $\nu$. 
 Adding the assumption that $W_p(\delta_x,m_x^t)\rightarrow 0$ as $t\rightarrow 0$ uniformly, 
 it follows that the family $\{\tilde{\nu}^t\}_{t>0}$ is a continuous curve in $\Pro_p(\Pro_p(X))$. 
\end{rem}
\begin{rem}
Let $(M,g)$ be a compact Riemannian manifold with the heat kernel $\{p^t_x\}_{x\in X}$. 
Suppose that the Ricci curvature is bounded below by a constant $K>0$. 
Then the same conclusion as above holds, since $W_p(p^t_x,p^t_y)\leq e^{-Kt}d(x,y)$ holds 
for all $t\geq 0$, $x,y\in M$ and $1\leq p\leq\infty$ 
(see \cite{VS}). 
\end{rem}
\section{Coarse Ricci curvature and Gromov-Hausdorff topology}\label{convres1}
 We define the notion of \emph{the Gromov-Hausdorff convergence}.
 \begin{defn}\label{GHCONV}
  Let $X$ and $Y$ be two metric spaces. 
  We call a map $f : X\rightarrow Y$ an \emph{$\epsilon$-approximation map} 
  if the following two conditions are satisfied : 
  \begin{align}
   |d_Y(f(x),f(y))-d_X(x,y)|\leq\epsilon \quad \text{for any}\; x,y\in X,\label{almostisom}\\
   Y\subset B_{\epsilon}\left(f(X)\right):=\left\{\;y\in Y\; |\; d_Y(f(X),y)\leq\epsilon\;\right\}.\label{almostonto}
  \end{align}
  Let $\{X_n\}_{n\in\mathbb{N}}$ be a sequence of 
  compact metric spaces and $X$ a compact metric space. 
  We say that \emph{$X_n$ converges to $X$ as $n\rightarrow \infty$ 
  in the sense of the Gromov-Hausdorff topology} 
  if there exist a decreasing sequence of 
  positive numbers $\{\epsilon_n\}_{n\in\mathbb{N}}$ tending to $0$ 
  and maps $f_n : X_n\rightarrow X$ such that 
  $f_n$ is an $\epsilon_n$-approximation map for any $n$.  
 \end{defn}
\begin{rem}
 For any $\epsilon$-approximation map, there exists a Borel measurable 
 $\epsilon$-approximation map close to the original map.  
 Hence we always assume that an approximation map is Borel measurable in this paper. 
\end{rem} 
 \begin{rem}\label{GHanother}
 Let $X_n$, $n=1,2,\ldots$, and $X$ be compact metric spaces. 
 We see that $X_n$ converges to $X$ 
 in the sense of the Gromov-Hausdorff topology if and only if 
 there exist a compact metric space $Z$ and 
 isometric embeddings $\phi_n\,:\,X_n\rightarrow Z$, $\phi\,:\,X\rightarrow Z$ such that 
 $d_H(\phi_n(X_n),\phi(X))\rightarrow 0$ as $n\rightarrow \infty$, where $d_H$ is 
 the Hausdorff distance on $Z$ (see \cite{BBI}). 
 \end{rem}
\begin{rem}\label{approrem}
  Let $(X,d_X),(Y,d_Y)$ be two metric spaces and 
  $f$ an $\epsilon$-approximation map from $X$ to $Y$. 
  Then there exists a $3\epsilon$-approximation map $f'$ from $Y$ to $X$ such that 
  \begin{align}
  d_Y(y, f(f'(y)))\leq \epsilon\label{eqinrem}
  \end{align}
  and 
  \begin{align}
   d_X(x, f'(f(x)))\leq 2\epsilon\label{eq2inrem}
  \end{align}
  holds (see \cite{LV}). 
 \end{rem}
 A metric space $X$ is said to be \emph{proper} if any bounded closed set in $X$ is compact. 
 It is known that a locally compact complete geodesic metric space is proper. 
 In this paper, we denote by $B_r(x)$ the closed ball centered at $x$ and of radius $r$. 
 We define the \emph{pointed Gromov-Hausdorff convergence} of proper metric spaces. 
 \begin{defn}[cf.\!\! \cite{BBI}*{Definition 8.1.1}]\label{ptGH}
  Let $\{(X_n,\ast_n)\}_{n\in\mathbb{N}}$ be a sequence of proper pointed metric spaces and 
  $(X,\ast)$ a proper pointed metric space. 
  We say that $(X_n,\ast_n)$ converges to $(X,\ast)$ as $n\rightarrow \infty$ 
  in the \emph{pointed Gromov-Hausdorff topology} if 
  for any $R>0$, there exist a sequence $\epsilon_n\rightarrow 0$ and pointed 
  $\epsilon_n$-approximation maps $f_n\,:\,B_{R+\epsilon_n}(\ast_n)\rightarrow B_R(\ast)$, 
  where a pointed approximation map means an approximation map with $f_n(\ast_n)=\ast$. 
 \end{defn}
 Let $x\in X$ and $x_n\in X_n$, $n=1,2\ldots$, be points. We say that $x_n$ \emph{converges to} $x$ 
 as $n\rightarrow\infty$ if 
  $X_n\rightarrow X$ in the pointed Gromov-Hausdorff topology and if $f_n(x_n)\rightarrow x$ in $X$, 
  where $f_n$ is an $\epsilon_n$-approximation map with $\epsilon_n\rightarrow 0$. 
 \begin{rem}
  Combining Definition \ref{ptGH} and the assumption that 
  $X_n$ and $X$ are all geodesic metric spaces 
  implies the following condition (see \cite{BBI}*{Exercise 8.1.4}).  
   There exist a sequence $R_l\rightarrow\infty$, a sequence 
   $\epsilon_n\rightarrow 0$ and the pointed $\epsilon_n$-approximation map  
   $f^l_n\,:\,B_{R_l}(\ast_n)\rightarrow B_{R_l}(\ast)$.
 \end{rem}
 We also define the \emph{local Gromov-Hausdorff convergence}, 
 which is a notion of the convergence of a family of metric spaces 
 being even not locally compact. 
 \begin{defn}[\cite{V2}*{Definition 27.11}]
  Let $\{X_n\}_{n\in\mathbb{N}}$ be a family of geodesic Polish spaces and $X$ a Polish space. 
  We say that $X_n$ converges to $X$ in the \emph{local Gromov-Hausdorff topology} if 
  there exist nondecreasing sequences of compact sets $\{K^{(l)}_n\}_{l\in\mathbb{N}}$ in each $X_n$ 
  and $\{K^{(l)}\}_{l\in \mathbb{N}}$ in $X$ such that the following 
  two conditions are satisfied: 
  \begin{enumerate}
   \item $\bigcup K^{(l)}$ is dense in $X$; 
   \item For each fixed $l\in\mathbb{N}$, $K^{(l)}_n$ converges to $K^{(l)}$ in the Gromov-Hausdorff sense as 
   $n\rightarrow\infty$. 
  \end{enumerate}
 \end{defn} 
 \begin{prop}\label{GHconv}
  Let $(X_n,\ast_n),\,n=1,2\ldots$ and $(X,\ast)$ be locally compact pointed Polish geodesic 
  metric spaces. 
  Assume that $(X_n,\ast_n)$ converges to $(X,\ast)$ in the pointed Gromov-Hausdorff topology. Then 
  $\Pro_p(X_n)$ converges to $\Pro_p(X)$ in the local Gromov-Hausdorff topology. 
 \end{prop}
 \begin{proof}
  Proposition \ref{GHconv} is proved in a similar way as in Theorem 28.13 in \cite{V2}. 
  For the completeness we show the detail of the proof. 
  Let $R_l\rightarrow \infty$ be an increasing sequence of positive numbers. 
  Without loss of generality, we assume $R_l>2$ for any $l\in\mathbb{N}$. 
  Define
  \begin{align}
   K^{(l)}:=\Pro_p(B_{R_l}(\ast))\subset \Pro_p(X),\notag\\
   K^{(l)}_n:=\Pro_p(B_{R_l}(\ast_n))\subset \Pro_p(X_n). \notag
  \end{align} 
  Since $B_{R_l}(\ast)$ is compact, so is $K^{(l)}$. 
  We take any $\mu\in\Pro_p(X)$ and fix it. Take $l_0>0$ with 
  $\mu(B_{R_{l_0}}(\ast))\geq 1/2$. 
  Set $\mu_l=\chi_{B_{R_l}(\ast)}\mu/\mu(B_{R_l}(\ast))$ for $l\geq l_0$, 
  where $\chi_A$ is the characteristic function of $A\subset X$. 
  For any $f\in\mathcal{C}_b(X)$, we have 
  \begin{align}
  &\left|\int_Xf\,d\mu_l-\int_Xf\,d\mu\right|
  =\left|\frac{1}{\mu(B_{R_l}(\ast))}\int_{B_{R_l}(\ast)}f\,d\mu-\int_Xf\,d\mu\right|\notag\\
  &\leq \left(\frac{1}{\mu(B_{R_l}(\ast))}-1\right)\left|\int_Xf\,d\mu\right|
  +\frac{1}{\mu(B_{R_l}(\ast))}\int_{X\setminus B_{R_l}(\ast)}|f|\,d\mu\notag\\
  &\leq 4\mathrm{sup}_{x\in X}|f|\mu(X\setminus B_{R_l}(\ast))
  \rightarrow 0\quad \text{as}\;l\rightarrow \infty.\notag
  \end{align}
  This implies that $\mu_l\rightarrow \mu$ weakly. 
  Let us prove that, for any $\epsilon>0$, there exists $l_1\in\mathbb{N}$ such that 
  \begin{align}
   \int_Xd(\ast,x)^p\,\mu_l(dx)\leq \int_Xd(\ast,x)^p\,\mu(dx)+\epsilon\notag
  \end{align}
  holds whenever $l\geq l_1$. 
  Indeed, let $M:=\int_Xd(\ast,x)^p\,\mu(dx)$ and let $\eta$ be a real number such that 
  $0<\eta<\epsilon/M$. 
  Take $l_1\in\mathbb{N}$ such that $1/\mu(B_{R_{l_1}}(\ast))\leq 1+\eta$. 
  For any $l\geq l_1$, we have 
  \begin{align}
   &\int_Xd(\ast,x)^p\,\mu_l(dx)=\frac{1}{\mu(B_{R_l}(\ast))}\int_{B_{R_l}(\ast)}d(\ast,x)^p\mu(dx)\notag\\
   &\leq \frac{1}{\mu(B_{R_l}(\ast))}\int_Xd(\ast,x)^p\,\mu(dx)\leq (1+\eta)\int_Xd(\ast,x)^p\,\mu(dx)\notag\\
   &<\int_Xd(\ast,x)^p\,\mu(dx)+\epsilon. \notag
  \end{align} 
  This means that the condition (\ref{Wpequiv}) in Theorem \ref{equiv} holds. 
  Therefore, $W_p(\mu_l,\mu)\rightarrow 0$ as $l\rightarrow \infty$, 
  so that $\cup K^{(l)}$ is dense in $\Pro_p(X)$. 
  It suffices to prove that $K_n^{(l)}\rightarrow K^{(l)}$ in the Gromov-Hausdorff sense. 
  For given $0<\epsilon<1/10$, there exist an $\epsilon$-approximation map 
  $f\,:\,B_{R_l}(\ast_n)\rightarrow B_{R_l}(\ast)$ and a $3\epsilon$-approximation map 
  $f'\,:\,B_{R_l}(\ast)\rightarrow B_{R_l}(\ast_n)$ as in Remark \ref{approrem} for sufficiently large $n$. 
  Let $\mu,\nu\in K^{(l)}_n$. For an optimal coupling $\pi_1\in\Pi(\mu,\nu)$ between 
  $\mu$ and $\nu\in K_n^{(l)}$, the push forward measure $\pi_2:=(f\times f)_{\ast}\pi_1$ 
  is a coupling between 
  $f_{\ast}\mu$ and $f_{\ast}\nu\in K^{(l)}$. 
  Let 
  \begin{align}
  A&:=\{\,(x,y)\in B_{R_l}(\ast_n)\times B_{R_l}(\ast_n)\,;\,d(x,y)\geq \epsilon^{1/2} 
  R_l/2\,\},\notag\\
  B&:=B_{R_l}(\ast_n)\times B_{R_l}(\ast_n)\setminus A.\notag
  \end{align}
  We get 
  \begin{align}
   &W_p(f_{\ast}\mu,f_{\ast}\nu)^p\notag\\
   &\leq \int_{B_{R_{l}}(\ast)\times B_{R_{l}}(\ast)}d(y_1,y_2)^p\,\pi_2(dy_1,dy_2)\notag\\
   &=\int_{B_{R_{l}}(\ast_n)\times B_{R_{l}}(\ast_n)}d(f(x_1),f(x_2))^p\,\pi_1(dx_1,dx_2)\notag\\
   &=\int_Ad(f(x_1),f(x_2))^p\,\pi_1(dx_1,dx_2)+\int_Bd(f(x_1),f(x_2))^p\,\pi_1(dx_1,dx_2)\notag\\
   &\leq\int_A(d(x_1,x_2)+\epsilon)^p\,\pi_1(dx_1,dx_2)+
   \int_B(d(x_1,x_2)+\epsilon)^p\,\pi_1(dx_1,dx_2) \notag\\
   &=\int_Ad(x_1,x_2)^p(1+\epsilon/d(x_1,x_2))^p\,\pi_1(dx_1,dx_2)+
   \int_B(d(x_1,x_2)+\epsilon)^p\,\pi_1(dx_1,dx_2) \notag\\
   &\leq\int_Ad(x_1,x_2)^p(1+p\epsilon^{1/2}/R_l)\,\pi_1(dx_1,dx_2)+
   \int_B\epsilon^{p/2}\left(\frac{R_l}{2}+\epsilon^{1/2}\right)^p\,\pi_1(dx_1,dx_2)\notag\\
   &\leq W_p(\mu,\nu)^p+O(\epsilon^{1/2}).\notag  
  \end{align} 
  Then we have 
  \begin{align}
   W_p(f_{\ast}\mu,f_{\ast}\nu)\leq W_p(\mu,\nu)+O(\epsilon^{1/2p}).\label{eq1inprop}
  \end{align}
  The same argument leads to
  \begin{align}
   W_p(f'_{\ast}(f_{\ast}\mu),f'_{\ast}(f_{\ast}\nu))\leq W_p(f_{\ast}\mu,f_{\ast}\nu)+
   O(\epsilon^{1/2p}).\label{eq2inprop}
  \end{align}
  Using (\ref{eq2inrem}), we also get 
  \begin{align}
   W_p((f'\circ f)_{\ast}\mu,\mu)&\leq
   \left\{\int_{B_{R_{l}}(\ast_n)}d(f'(f(x)),x)^p\,\mu(dx)\right\}^{1/p}\notag\\
   &\leq 2\epsilon\notag
  \end{align}
  and 
  \begin{align}
   W_p((f'\circ f)_{\ast}\nu,\nu)\leq 2\epsilon.\notag
  \end{align}
  By the triangle inequality and (\ref{eq2inprop}), we get
  \begin{align}
   W_p(\mu,\nu)\leq W_p(f_{\ast}\mu,f_{\ast}\nu)+O(\epsilon^{1/2p}).\label{eq3inprop}
  \end{align}
  The inequalities (\ref{eq1inprop}) and (\ref{eq3inprop}) imply the condition (\ref{almostisom}) 
  for $f_{\ast} : K_n^{(l)}\rightarrow K^{(l)}$. 
  Moreover by (\ref{eqinrem}),
  \begin{align}
   W_p(\sigma, (f\circ f')_{\ast}\sigma)\leq\epsilon\label{eq4inprop}
  \end{align}
  holds for any $\sigma\in K^{(l)}$. The inequality (\ref{eq4inprop}) implies the condition 
  (\ref{almostonto}) for $f_{\ast} : K_n^{(l)}\rightarrow K^{(l)}$. 
  Then $f_{\ast}\,:\,K^{(l)}_n\rightarrow K^{(l)}$ is an 
  $O(\epsilon^{1/2p})$-approximation map. This completes the proof of 
  Proposition \ref{GHconv}.  
  \end{proof} 
  We define a notion of convergence of maps. 
  We omit the base point of a pointed metric space if there is no confusion. 
\begin{defn}\label{convasmap}
 Let $X,Y,X_n,Y_n,\,n=1,2,\ldots$, be proper pointed metric spaces 
 and $\phi_n\,:X_n\rightarrow Y_n$ maps. 
 Suppose that $X_n\rightarrow X$ and $Y_n\rightarrow Y$ in the pointed 
 Gromov-Hausdorff topology. 
 Let $f_n:X_n\rightarrow X,\, g_n:Y_n\rightarrow Y$ be approximation maps. 
 We say that the sequence of maps $\{\phi_n\}_n$ converges to $\phi\,:\,X\rightarrow Y$ if 
 for any sequence $\{x_n\}_{n\in \mathbb{N}}$ with $f_n(x_n)\rightarrow x$, we have 
 $g_n(\phi_n(x_n))\rightarrow \phi(x)$. 
\end{defn}
  \begin{proof}[Proof of Theorem \ref{GHtop}]
   We use the same notation as in the proof of Proposition \ref{GHconv}. 
   We assume $R_{l+1}-R_l>1$ for any $l\geq 1$ without loss of generality. 
   Let $X_n\ni x_n\rightarrow x\in X$ as $n\rightarrow \infty$ and let 
   $f_{j,n}\,:\,B_{R_j}(\ast_n)\rightarrow B_{R_j}(\ast)$ be an  
   $\epsilon_n$-approximation map for a sequence of real numbers $0<\epsilon_n<1$ tending to $0$. 
   For any $\epsilon>0$, there exist compact subsets $L_{\epsilon}^n\subset X_n$ with 
   $x_n\in L^n_{\epsilon}$ such that 
   $m^n_{x_n}(X_n\setminus L^n_{\epsilon})\leq\epsilon$ and 
   $\sup_n\Diam\, L^n_{\epsilon}=r_{\epsilon}<\infty$ by the assumption of Theorem \ref{GHtop}. 
   Then $L_{\epsilon}^n\subset B_{r_{\epsilon}+d(x_n,\ast_n)}(\ast_n)\subset 
   B_{r_{\epsilon}+d(x,\ast)+1}(\ast_n)$ holds for sufficiently large $n$. 
   Let $\{\mu^n_j\in K^{(j)}_n\}_{j}$ be an approximation of $m^n_{x_n}$ 
   as in Proposition \ref{GHconv}, that is, 
   $\mu_j^n=\chi_{B_{R_j}(\ast_n)}m^n_{x_n}/m^n_{x_n}(B_{R_j(\ast_n)})$. 
   By taking $i,j$ such that $R_j>R_i\geq r_{\epsilon}+d(x,\ast)+1$, 
   we get $L^{n}_{\epsilon}\subset B_{R_i}(\ast_n)\subset B_{R_j}(\ast_n)$. 
   Then we have 
   \begin{align}
    (f_{j,n})_{\ast}\mu^n_j(X\setminus B_{R_{i+1}}(\ast))
    &\leq\mu^n_j(B_{R_j}(\ast_n)\setminus B_{R_i}(\ast_n))\notag\\
    &=1-\mu^n_j(B_{R_i}(\ast_n))\notag\\
    &\leq 1-\mu^n_j(L^n_{\epsilon})\notag\\
    &= 1-\frac{1}{m^n_{x_n}(B_{R_j}(\ast_n))}
    m^n_{x_n}(B_{R_j}(\ast_n)\cap L^n_{\epsilon})\notag\\
    &\leq 1-m^n_{x_n}(L^n_{\epsilon})\notag\\
    &=m^n_{x_n}(X\setminus L^n_{\epsilon})\notag\\
    &\leq \epsilon. \notag
   \end{align}
   This means that the family of probability measures 
   $\{(f_{j,n})_{\ast}\mu_j^n\}_{n\in\mathbb{N}}$ is tight. 
   Set $f_n:=f^n_{j,n}$ for simplicity. 
   By extracting a subsequence of $\{(f_n)_{\ast}\mu_j^n\}_{n}$ 
   (we denote it by $k:=n_k$), we have a probability measure 
   $\mu_j(x)\in K^{(j)}$ with $(f_k)_{\ast}\mu^k_j\rightarrow \mu_j(x)$ 
   weakly by using Proposition \ref{proh}. 
   We may take a sufficiently large $j$ 
   that satisfies $1/m^k_{x_k}(B_{R_j}(\ast_k))\leq 2$. 
   Since the diameter of $L^k_{\epsilon}\subset B_{r_{\epsilon}+d(x,\ast)+1}(\ast_k)$ is 
   independent of $k$, 
   we are able to take such $j$ being independent of the choice of $k$ by $(2)$ of the assumption 
   of the theorem.
   Then, for given $R>0$, we have 
   \begin{align}
    &\int_{d(\ast,y)\geq R}d(\ast,y)^p\,(f_k)_{\ast}\mu_j^k(dy)
    \leq\int_{d(f_k(\ast_k),f_k(y))\geq R-\epsilon_k}d(f_k(\ast_k),f_k(y))^p\,\mu^k_j(dy)\notag\\
    &=\frac{1}{m^k_{x_k}(B_{R_j}(\ast_k))}\int_{R-\epsilon_k\leq d(f_k(\ast_k),f_k(y))\leq R_j+\epsilon_k}
    d(f_k(\ast_k),f_k(y))^p\,m^k_{x_k}(dy)\notag\\
    &\leq 2\int_{R-\epsilon_k\leq d(f_k(\ast_k),f_k(y))\leq R_j+\epsilon_k}
    d(f_k(\ast_k),f_k(y))^p\,m^k_{x_k}(dy)\notag\\
    &\leq 2\int_{R-2\epsilon_k\leq d(\ast_k,y)}
    (d(\ast_k,y)+\epsilon_k)^p\,m^k_{x_k}(dy).\notag
   \end{align}
   Hence by $(3)$ of the assumption of the theorem, 
   \begin{align}
    \lim_{R\rightarrow\infty}\limsup_{k\rightarrow\infty}
    \int_{d(\ast,y)\geq R}d(\ast,y)^p\,(f_k)_{\ast}\mu^k_j(dy)
    =0\label{Wpconv}
   \end{align}
   holds. We conclude $W_p(\mu_j(x),(f_k)_{\ast}\mu_j^k)\rightarrow 0$ by Theorem \ref{equiv}. 
   By (\ref{lsc}), we get 
   \begin{align}
    \mu_j(x)(X\setminus B_{R_{i+1}}(\ast))
    &\leq\liminf_{k\rightarrow\infty}(f_k)_{\ast}\mu^k_j(X\setminus B_{R_{i+1}}(\ast))\notag\\
    &\leq\liminf_{k\rightarrow\infty}\mu^k_j(X_k\setminus B_{R_i}(\ast_k))
    \leq 2\epsilon.\notag
   \end{align}
   Since the sequence $\{\mu_j(x)\}_j$ is tight, 
   there exists a probability measure $\mu_x\in\Pro(X)$ such that, by extracting subsequence, 
   $\mu_j(x)\rightarrow \mu_x$ weakly by Proposition \ref{proh}. 
   A similar argument of (\ref{Wpconv}) leads to 
   $W_p(\mu_j(x),\mu_x)\rightarrow 0$. 
   For the dense subset $D:=\{x^i\}_i$ of $X$, we are able to take families of dense subset $\{x^i_k\}_i$ 
   of $X_k$ such that $f_k(x^i_k)\rightarrow x^i$. By using the above argument, we define the map 
   $\mu\,:\,D\ni x^j\rightarrow \mu_{x^j}\in\Pro_p(X)$. 
   Take $x,y\in D$. Let $x_k,y_k\in X_k$ be convergence sequences such that  
   $x_k\rightarrow x$ and $y_k\rightarrow y$. 
   For any $\epsilon>0$, 
   there exist sufficiently large $i_x,i_y\in\mathbb{N}$ such that 
   \begin{align}
    W_p(\mu^{x_k}_i,m^k_{x_k})\leq \epsilon\quad\text{for any}\,i\geq i_x,\notag\\
    W_p(\mu^{y_k}_i,m^k_{x_k})\leq \epsilon\quad\text{for any}\,i\geq i_y,\notag
   \end{align} 
    where $\mu_i^{x_k}$ is an approximation measure for $m^k_{x_k}$ as in Proposition \ref{GHconv} 
    and we are able to take both $i_x,i_y$ independent of $k$. 
    Then we have 
    \begin{align}
     &W_p(\mu_x,\mu_y)=\lim_{i\rightarrow\infty}W_p(\mu_i(x),\mu_i(y))\notag\\
     &=\lim_{i\rightarrow\infty}\lim_{k\rightarrow\infty}
     W_p((f_k)_{\ast}\mu^{x_k}_i,(f_k)_{\ast}\mu^{y_k}_i)\notag\\
     &\leq \lim_{i\rightarrow\infty}\limsup_{k\rightarrow\infty}W_p(\mu^{x_k}_i,\mu^{y_k}_i)\notag\\
     &\leq \lim_{i\rightarrow\infty}\limsup_{k\rightarrow\infty}
     \left(W_p(\mu^{x_k}_i,m^k_{x_k})+W_p(m_{x_k}^k,m_{y_k}^k)+
     W_p(m^k_{y_k},\mu^{y_k}_i)\right)\notag\\
     &\leq \lim_{i\rightarrow\infty}\limsup_{k\rightarrow\infty}
     \left(W_p(\mu^{x_k}_i,m^k_{x_k})+(1-\kappa_0)d(x_k,y_k)+
     W_p(m^k_{y_k},\mu^{y_k}_i)\right)\notag\\
     &\leq (1-\kappa_0)d(x,y).\notag
    \end{align}
    We are able to define a map $m\,:\,X\rightarrow \Pro_p(X)$ that is a continuous extension of 
    the map $\mu$. This completes the proof of 
    Theorem \ref{GHtop}.
  \end{proof}
Ollivier defined a notion of the Gromov-Hausdorff convergence with random walks. 
\begin{defn}[cf. \cite{O}*{Definition 55}]
 Let $(X,\{m_x\}_{x\in X})$ and $(X^n,\{m^n_x\}_{x\in X^n})$, $n=1,2,\ldots$, be compact 
 metric spaces with 
 random walks. We say that 
 $(X^n,\{m_x^n\}_{x\in X^n})$ \emph{converges to} $(X,\{m_x\}_{x\in X})$ 
 if for any $\epsilon>0$, there exists $N_{\epsilon}\in\mathbb{N}$ such that 
 the following conditions (1) and (2) are satisfied: 
 \begin{enumerate}  
 \item There exist a compact metric space $(Z,d)$ and isometric embeddings 
 $\phi_n\,:\,X^n\rightarrow Z$, $\phi\,:\,X\rightarrow Z$.  
 \item For any $x\in X$ there exists $x_n\in X^n$ with $d(\phi_n(x_n),\phi(x))\leq \epsilon$ such that 
 $W_p((\phi_n)_{\ast}m^n_{x_n},(\phi)_{\ast}m_x)\leq 2\epsilon$ whenever $n\geq N_{\epsilon}$, 
 and likewise $x_n\in X^n$. 
 \end{enumerate}
\end{defn}
It is clear that $X^n$ Gromov-Hausdorff converges to $X$ if 
$(X^n,\{m^n_x\}_x)$ converges to $(X,\{m_x\}_x)$ 
as long as $X^n$ and $X$ are compact. 
\begin{prop}\label{OGH} 
$(X^n,\{m^n_x\}_{x\in X^n})$ converges to $(X,\{m_x\}_{x\in X})$ 
is equivalent to  
$m^n\rightarrow m$ as a map provided that the $p$-coarse Ricci curvature 
is bounded below uniformly, and $X^n$, $X$ are compact. 
\end{prop}
\begin{proof}
 Suppose that $(X^n,d^n,\{m_x^n\}_{x\in X^n})\rightarrow (X,d,\{m_x\}_{x\in X})$ in Ollivier's sense 
 and there exists a uniform lower bound of the $p$-coarse Ricci curvature $\kappa_0\in\mathbb{R}$. 
 We fix an arbitrary positive constant $\epsilon>0$. Assume $x_n\in X^n$ converges to $x\in X$ 
 in the sense of the Gromov-Hausdorff topology. 
 Then there exists $N_{\epsilon}>0$ such that 
 $d(\phi(x),\phi_n(x_n))\leq \epsilon$ for any $n\geq N_{\epsilon}$, where $\phi, \phi_n$ are isometric 
 embeddings into a compact metric space. 
 At the same time, there exists $x'_n\in X^n$ such that $d(\phi(x),\phi_n(x'_n))\leq \epsilon$ and 
 $W_p(\phi_{\ast}m_x,(\phi_n)_{\ast}m^n_{x'_n})\leq 2\epsilon$ for any $n\geq N_{\epsilon}$ 
 by the definition of the convergence of metric spaces with random walks. 
 Then we have 
 \begin{align}
  W_p(\phi_{\ast}m_x,(\phi_n)_{\ast}m^n_{x_n})&\leq W_p(\phi_{\ast}m_x,(\phi_n)_{\ast}m^n_{x'_n})
  +W_p((\phi_n)_{\ast}m^n_{x'_n},(\phi_n)_{\ast}m^n_{x_n})\notag\\
  &\leq 2\epsilon+W_p(m^n_{x'_n},m^n_{x_n})\notag\\
  &\leq 2\epsilon+(1-\kappa_0)d(x'_n,x_n)\notag\\
  &=2\epsilon+(1-\kappa_0)d(\phi_n(x'_n),\phi_n(x_n))\notag\\
  &\leq 2\epsilon+(1-\kappa_0)\left(d(\phi_n(x'_n),\phi(x))+d(\phi(x),\phi_n(x_n))\right)\notag\\
  &\leq 2(2-\kappa_0)\epsilon.\notag 
 \end{align}
 Since $\epsilon$ is an arbitrary number, we get $m^n\rightarrow m$ as maps. 
 
 Suppose $m^n\rightarrow m$ as maps. 
 We fix $\epsilon>0$ and take a sufficiently large $n$ such that 
 $f_n\,:\,X^n\rightarrow X$ be an $\epsilon$-approximation map. 
 Let $\phi\,:\,X\rightarrow Z$ and $\phi_n\,:\,X^n\rightarrow Z$ be isometric embeddings 
 into a compact metric space $Z$. 
 It is easy to get $d(\phi_n(q),\phi(f_n(q)))\leq 2\epsilon$ for any $q\in X^n$. 
 By the assumption, 
 we are able to take $x\in X$ and $x_n\in X^n$ such that $d(\phi(x),\phi_n(x_n))\leq 2\epsilon$ and 
 $W_p(m_x,(f_n)_{\ast}m_{x_n})\leq 2\epsilon$. 
 Then 
 \begin{align}
  W_p(\phi_{\ast}m_x,(\phi_n)_{\ast}m^n_{x_n})
  &\leq W_p(\phi_{\ast}m_x,\phi_{\ast}(f_n)_{\ast}m^n_{x_n})
  +W_p(\phi_{\ast}(f_n)_{\ast}m^n_{x_n},(\phi_n)_{\ast}m^n_{x_n})\notag\\
  &=W_p(m_x,(f_n)_{\ast}m^n_{x_n})
  +W_p(\phi_{\ast}(f_n)_{\ast}m^n_{x_n},(\phi_n)_{\ast}m^n_{x_n})\notag
 \end{align}
 holds. We have 
 \begin{align}
  W_p(\phi_{\ast}(f_n)_{\ast}m^n_{x_n},(\phi_n)_{\ast}m^n_{x_n})^p
  =\int_{X_n}d(\phi(f_n(q)),\phi_n(q))^p\,m^n_{x_n}(dq)\leq (2\epsilon)^p.\notag
 \end{align}
 Then we obtain 
 \begin{align}
  W_p(\phi_{\ast}m_x,(\phi_n)_{\ast}m^n_{x_n})\leq 4\epsilon.\notag
 \end{align}
 This completes the proof. 
 \end{proof}
\section{Concentration of measure phenomenon}\label{convres2}
We show Theorem \ref{Levy} in this section. We call $(X,d,\nu)$ a \emph{metric measure space} 
 if $(X,d)$ is a complete separable metric space and $\nu\in\Pro(X)$.
\begin{defn}
 Let $(X,d,\mu)$ be a metric measure space. 
 We define the \emph{partial diameter} of $\mu$ by 
 \begin{align}
  \Diam(\mu, 1-\kappa):=\inf\left\{\Diam(A) ; \mu(A)\geq 1-\kappa,\; A\in\mathcal{B}(\R)\right\}.\label{pd}
 \end{align}
 We also define 
 \begin{align}
  &\Ob(X ; -\kappa):=\sup\left\{\Diam(f_{\ast}\mu, 1-\kappa) ; 
  f : X\rightarrow \R : \text{1-Lipschitz map}\right\}.\notag
 \end{align}
 We define the \emph{observable diameter} of $(X,d,\mu)$ to be 
 \begin{align}
  &\Ob(X):=\inf_{\kappa\in(0,1)}\max\left\{\Ob(X ;-\kappa),\kappa\right\}\label{Ob}
 \end{align}
\end{defn}
\begin{defn}\label{levydef}
 A sequence of metric measure spaces $\{X_n\}_{n\in\mathbb{N}}$ is called a \emph{L{\' e}vy family} if 
 \begin{align}
  \Ob(X_n)\rightarrow 0\quad \text{as}\quad n\rightarrow \infty.\label{levy}
 \end{align}
\end{defn}
\begin{rem}
 There exist various definitions of the L{\'e}vy family (see \cite{G,L}). 
\end{rem}
\begin{proof}[Proof of Theorem \ref{Levy}]
 Let $C$ be a uniform Lipschitz constant of maps $\{m_n\}_n$. 
 It suffices to prove the following two claims 
 \begin{align}
  &\Ob(CX; -\kappa)\leq C\Ob(X; -\kappa)\quad \text{for any $\kappa\in(0,1)$},\label{claim1}\\
  &\Ob(\Pro_1(X))\leq\Ob(CX),\label{claim2}
 \end{align}
 where $CX:=(X,Cd_X,\mu_X)$.  
 Indeed 
 \begin{align}
  \Ob(\Pro_1(X_n))\leq \Ob(CX_n)\leq C\Ob(X_n)\rightarrow 0\notag
 \end{align}
 provided (\ref{claim1}) and (\ref{claim2}) hold. 
 
 For any $\epsilon>0$ there exists a 1-Lipschitz map $f : X\rightarrow \R$ and 
 a Borel set $A\subset\R$ such that 
 \begin{align}
  &f_{\ast}\mu_X(A)\geq 1-\kappa\notag\\
  &\Diam(A)\geq \Ob(X; -\kappa)-\epsilon.\notag
 \end{align}
 We define $B:=(1/c)A=\{a/c ; a\in A\}$. Then we have 
 \begin{align}
  &\left(\frac{1}{c}f\right)_{\ast}\mu_X(B)=\mu_X(A)\notag\\
  &\Diam(B)=\frac{1}{c}\Diam(A).\notag
 \end{align}
 Hence $\Ob(CX; -\kappa)\leq C\Ob(X; -\kappa)$ holds for any $\kappa\in(0,1)$. 
 (\ref{claim1}) is satisfied.   
 Since 
 \begin{align}
 \{\Diam(f_{\ast}(m_{\ast}\mu_X); 1-\kappa); f : \Pro_1(X)\rightarrow\R : 1\text{-Lipschitz}\}\notag\\
 \subset \{\Diam(f_{\ast}\mu_X; 1-\kappa); f : CX\rightarrow\R : 1\text{-Lipschitz}\},\notag
 \end{align} 
 (\ref{claim2}) holds clearly. 
\end{proof}
\section*{Acknowledgement}
The author is grateful to Professor Takashi Shioya for reading this paper and giving useful comments, 
Professor Shin-ich Ohta for fruitful comments for Theorem \ref{lipext} and Ayato Mitsuishi 
for useful advice for Theorem \ref{GHtop}. 

\end{document}